\documentclass[a4paper]{article}

\usepackage[english]{babel}
\usepackage[utf8x]{inputenc}
\usepackage[T1]{fontenc}

\usepackage[a4paper,top=3cm,bottom=2cm,left=3cm,right=3cm,marginparwidth=1.75cm]{geometry}

\usepackage{amsmath}
\usepackage{graphicx}
\usepackage[colorinlistoftodos]{todonotes}
\usepackage[colorlinks=true, allcolors=blue]{hyperref}

\input{definition}
\newcommand{\model}{\Mcal}
\newcommand{\real}{\RR}
\newenvironment{proof}{\par\noindent{\bf Proof\ }}{\hfill\BlackBox\\[2mm]}

\title{Normalized Maximum Likelihood with Luckiness for Multivariate Normal Distributions}
\author{Kohei Miyaguchi\footnotemark}

\begin{document}
\maketitle
\footnotetext{Graduate Schoold of Information Science and Technology, The University of Tokyo\\
kohei\_miyaguchi@mist.i.u-tokyo.ac.jp}

\begin{abstract}
The normalized maximum likelihood (NML) is one of the most important distribution in coding theory and statistics.
NML is the unique solution (if exists) to the pointwise minimax regret problem.
However, NML is not defined even for simple family of distributions such as the normal distributions.
Since there does not exist any meaningful minimax-regret distribution for such case,
it has been pointed out that \emph{NML with luckiness} (LNML) can be employed as an alternative to NML.
In this paper, we develop the closed forms of LNMLs for multivariate normal distributions.
\end{abstract}

\section{Introduction}

The normalized maximum likelihood (NML) is known to be the minimax optimal code length for
data compression or statistical model selection
with respect to the MDL principle \cite{rissanen1984universal}.
However, for some distribution class including the normal distributions,
NML is not well-defined.
Restricting sample spaces can solve the problem \cite{hirai2011efficient},
but it has an intrinsic flaw that samples can fall outside of any valid restrictions.
We will introduce a treatment for this issue, NML with luckiness (LNML),
and give an exact formula for LNML of the multivariate normal distributions.

The rest of the paper is organized as follows:
We first introduce NML and its optimality in Section~2.
Then we present the notion of LNML as an extension of NML in Section~3.
The main theorems about LNMLs of multivariate normal distributions are shown in Section~4.
Finally, in Section~5,  we discuss on the choice of the hyperparameters and future perspective. 

\section{Normalized Maximum Likelihood (NML)}
Let $\model$ be a set of probability density (or mass) functions, which we call a \emph{(statistical) model}.
The NML relative to a statistical model $\model$ is given by
\begin{align*}
    \bar{p}_n(x^n)\eqdef\frac{\max_{p\in \model^n} p(x^n)}{C(\model^n)},
\end{align*}
where $x^n=\cbr{x_i}_{i=1}^{n}$ and $C(\cdot)$ denotes the normalizing constant
$C(\model^n)=\int \max_{p\in \model^n} p(x^n) d\mu(x^n)$.
Here $\model^n$ denotes the model of $n$ i.i.d. observations drawn from a density in model $\model$.
The normalizing constant $C(\cdot)$ is also referred to as
the capacity function
because of intuitions from the theory of channel capacity \cite{rissanen2012optimal}.
It is known that the NML is the unique distribution that
attains the minimax point-wise regret on description lengths
with respect to the model $\model$ \cite{shtar1987universal}
\begin{align*}
    \max_{p\in \model, x^n\in\mathcal{X}^n} R(\bar{p};p,x^n)
    &= \min_{q:\int q = 1} \max_{p\in\model, x^n\in\mathcal{X}^n} R(q;p,x^n),
\end{align*}
where $R(q;p,x^n)\eqdef \log \frac{p(x^n)}{q(x^n)}$, if the minimum exists.

However, not a few of NMLs are difficult to obtain their closed form,
and on the other hand demanding to compute numerically due to the integral of capacity function.
Though an asymptotic approximation up to the constant term has been given by \cite{rissanen1996fisher},
there are still two major problems.
One is the problem of small samples such that
the approximation is valid only for sufficiently large sample and that
it is not known how many samples we need for the accurate approximation.
The other is the problem of improper capacities such that
the integral of the capacity is not finite and that the NML does not exist.
One can sidestep the former problem exploiting
a formula (e.g. \cite{kontkanen2007linear}) for calculating exact NMLs.
However the latter problem is inevitable
since there does not exist any target values.

\section{NML with Luckiness (LNML)}
Since NMLs may be improper (e.g. their capacities are undefined),
a number of treatments such as
restriction on the range of data and restriction on the range of parameters
has been considered so far.
In a previous study \cite{hirai2011efficient}, on exponential families of distributions,
exact formulae for NMLs with a variable restricted range of data,
$Y(\eta)$, were presented.
However there is a flaw on such restriction technique that it is impossible to
design $Y(\eta)$ to include any unseen data.

NML with luckiness (LNML) \cite{grunwald2007minimum} is another approach for solving the problem of improper NMLs.
It is a general solution that includes as special cases
the restriction on the range of the parameters and
the conditional NML.
One of the biggest advantage of LNMLs over restriction of the range of data is
that it is well-defined regardless of the support of data.
A LNML density is relative to a model $\model$ and luckiness $\pi:\model\to(0,\infty)$,
\begin{align*}
    p_n^{\rm LNML}(x^n) \eqdef
    \frac{\max_{p\in \model^n}p(x^n)\pi(p)}{C(\model^n, \pi)},
\end{align*}
where $C(\model^n, \pi)=\int \max_{p\in \model^n}p(x^k)\pi(p)d\mu(x^n)$ denotes
the normalizing constant.
Here we assume that the maximum exists.
If $\pi(p)=1$ for all $p\in\model$, then the LNML coincides with the NML.
On the other hand one can avoid the problem of infinite integrals by
choosing $\pi$ properly. Note that there are two variants (LNML1 and LNML3) except the above LNML,
as known as LNML2, but here we focus on LNML2 for the sake of simplicity.
See \cite{grunwald2007minimum} for the detailed definitions and differences.

Another remarkable property of LNMLs is tilted minimax optimality.
The LNML uniquely achieves the minimax pointwise \emph{tilted regret},
\begin{align*}
    \max_{p\in \model, x^n\in\mathcal{X}^n} R_\pi(p_n^{\rm LNML};p,x^n)
    &= \min_{q:\int q=1} \max_{p\in \model, x^n\in\mathcal{X}^n} R_\pi(q;p,x^n),
\end{align*}
where $R_\pi(q;p,x^n)=\log\frac{p(x^n)\pi(p)}{q(x^n)}$ denotes a tilted regret function.
This is a straight extension of the previous minimax regret problem.
Constant $\pi(=c)$ implies zero subjectivity on the regret, hence
NML is the minimax optimal with respect to the completely objective regret $R_c(q;p,x^n)$.
On the other hand LNMLs are minimax optimal with respect to their subjective regrets $R_\pi(q;p,x^n)$.
Therefore luckiness $\pi$ embodies one's subjectivity on the true distribution
in minimax encoding, and it is necessary to avoid the infinite capacity problem.

Now let $\model^k_\pi$ be a set of non-negative functions $\myset{p^k(\cdot)\pi(p)}{p\in \model}$.
Then, LNML relative to $(\model, \pi)$ can be written with the same notation as ordinary NMLs,
\begin{align*}
    p_n^{\rm LNML}(x^n) = \bar{p}_n(x^n) = \max_{p\in \model^n_\pi}\frac{p(x^n)}{C(\model^n_\pi)}.
\end{align*}
This is a natural generalization of NMLs on unnormalized density functions.
Moreover, by extending the domain of $\pi$ to any probability densities
such that $\pi(p)=0\;(\forall p\notin\model)$,
$\pi$ can be seen equivalent to the model $\model_\pi$ itself.

Henceforth we omit $\pi$ and regard $\model=\model_\pi$ as the unnormalized statistical model.
Given any unnormalized models $\model$, one can recover their luckiness $\pi$;
Let $\pi_{\model}$ denote the extended luckiness function relative to model $\model$,
\begin{align*}
    \pi_\model(p)\eqdef\sup \myset{\alpha\in[0, 1]}{\alpha p\in\model\cup\cbr{0}},
\end{align*}
for all probability densities $p$.
Conversely, luckiness $\pi_\model$ also has sufficient information for recovering the original model $\model$.

\section{LNML for Multivariate Normal Distributions}
In this section we give two formulae for LNMLs
relative to the $m$-dimensional normal distributions with
both mean $\mu\in\mathbb{R}^m$ and variance $\Sigma(>0)\in\mathbb{R}^{m\times m}$ unknown,
whose density function is as follows:
\begin{align*}
    f(x^n;\mu,\Sigma) &\eqdef
    \frac1{(2\pi)^{\frac{mn}2}|\Sigma|^{\frac{n}2}} \exp\cbr{
        -\frac1{2}\sum_{i=1}^{n} (x_i-\mu)^\top\Sigma^{-1}(x_i-\mu)
    }.
\end{align*}
Here, $\Sigma>0$ denotes that $\Sigma$ is positive definite.
First, we start from a simple case,
where the luckiness is given by
\begin{align}
    \pi(\mu, \Sigma;\nu,\sigma^2,\rho^2)=\frac1{(2\pi)^{\frac{m\nu}2}|\Sigma|^{\frac{\nu}2}}\exp\sbr{
        -\frac{\nu}{2}\tr\cbr{\Sigma^{-1} (\sigma^2I_m +\rho^2 \mu\mu^\top)}
    }, \label{eq:luckiness}
\end{align}
for $\nu > m-1$ and $\sigma^2,\rho^2 >0$.

\begin{theorem}[Simple luckiness]
    Let $\model^n$ be the model of $n$ independent observations drawn from an $m$-dimensional normal distribution
    with luckiness $\pi$ given by (\ref{eq:luckiness}).
    Then, the capacity of $\model^n$ is calculated as
    \begin{align*}
        C(\model^n) &= \mathcal{C}(m, n, \nu, \sigma^2, \rho^2)\\
        &\eqdef \frac1{\sigma^{m\nu}}
        \rbr{\frac{(n+\nu)^{n+\nu}\rbr{1+\frac{n}{\rho^2\nu}}}{(2e)^{n+\nu}(\pi\nu)^\nu}}^{\frac{m}2}
        \frac{\Gamma_m(\frac{\nu}2)}
        {\Gamma_m(\frac{n+\nu}2)},
    \end{align*}
    where $\Gamma_m(z)=\pi^{\frac{m(m-1)}4}\prod_{j=0}^{m-1}\Gamma(z-\frac{j}2)$ denotes
    the multivariate gamma function.
    Correspondingly, the LNML relative to $\model^n$ is given by
    \begin{align*}
        \bar{p}_n(x^n)
        &= \frac
        {f(x^n;\bar{\mu}_n,\bar\Sigma_n)
        \pi(\bar{\mu}_n,\bar\Sigma_n;\nu,\sigma^2,\rho^2)}
        {\mathcal{C}(n, m, \nu, \sigma^2,\rho^2)}\\
        &=
        \rbr{
            \frac
            {\nu^{\nu}}
            {\pi^n(n+\nu)^{n+\nu}(1+\frac{n}{\rho^2\nu})}
        }^{\frac{m}2}
        \frac
        {\Gamma_m(\frac{n+\nu}2)}
        {\Gamma_m(\frac{\nu}2)}
        \frac{\sigma^{m\nu}}{|\bar\Sigma_n|^{\frac{n+\nu}2}},
    \end{align*}
    where $\bar{\mu}_n=\frac{1}{n+\rho^2 \nu}\sum_{i=1}^{n} x_i$ and
    $\bar{\Sigma}_n=\frac1{n+\nu}(\sum_{i=1}^{n} x_ix_i^\top +\nu\sigma^2I_m)-\frac{1}{(\nu+n)(\rho^2\nu+n)}\sum_{i,j=1}^{n}x_ix_j^\top$
    denote the maximum a-posteriori probability (MAP) estimates under the prior proportional to $\pi$.
    \label{thm:LNML_basic}
\end{theorem}
\begin{proof}
    It is suffice to show the closed form of the capacity $C(\model^n)$.
    By the definition of the capacity, we have
    \begin{align}
        C(\model^n)
        &= \int \max_{\mu\in\real^m,\Sigma>0}
        f(x^n;\mu, \Sigma)\pi(\mu,\Sigma;\nu,\sigma^2, \rho^2) dx^n
        \nonumber
        \\&= \int f(x^n;\bar\mu_n, \bar\Sigma_n)\pi(\bar\mu_n,\bar\Sigma_n;\nu,\sigma^2,\rho^2) dx^n
        \nonumber
        \\&= \frac1{(2\pi e)^{\frac{m(\nu+n)}2}}\int \frac1{\abs{\bar{\Sigma}_n}^{\frac{\nu+n}2}} dx^n.
        \label{eq:capacity_simple_reduction}
    \end{align}
    On the other hand, we have a recursive expression of $\bar\Sigma_n$ such that
    \begin{align*}
      \bar\Sigma_n &= \frac{\nu+n-1}{\nu+n}\bar\Sigma_{n-1}
      + \frac{\rho^2\nu+n-1}{(\nu+n)(\rho^2\nu+n)}(x_n-\bar\mu_{n-1})(x_n-\bar\mu_{n-1})^\top.
    \end{align*}
    Therefore, using the matrix determinant lemma,
    the last integral \eqref{eq:capacity_simple_reduction} satisfies the following recursive formula,
    \begin{align*}
        \int \frac1{\abs{\bar{\Sigma}_n}^{\frac{\nu+n}2}} dx^n
        &=
        \int \frac1{\abs{\frac{\nu+n}{\nu+n}\bar{\Sigma}_n
        }^{\frac{\nu+n}2}
        } dx^n
        \\&=
        \int \frac1{\abs{\frac{\nu+n-1}{\nu+n}\bar{\Sigma}_{n-1}}^{\frac{\nu+n}2}}
        \frac{1}{\cbr{1+c_{n-1}(x_{n}-\bar{\mu}_{n-1})^\top \bar\Sigma_{n-1}^{-1}(x_{n} -\bar{\mu}_{n-1})}^{\frac{\nu+n}2}}
        dx^n
        \\&=
        \frac{\pi^{\frac{m}2} \Gamma(\frac{\nu+n-m}2)
        }{\rbr{\frac{\nu+n-1}{\nu+n}c_{n-1} }^{\frac{m}2}\Gamma(\frac{\nu+n}2)}
        \int \frac{1}{\abs{\frac{\nu+n-1}{\nu+n}\bar\Sigma_{n-1}}^{\frac{\nu+n-1}2}}
        dx^{n-1},
    \end{align*}
    where $c_k=\frac{\rho^2\nu+k}{(\nu+k)(\rho^2\nu+k+1)}$.
    In the last equation, we exploits the normalizing factor of the multivariate $t$-distributions.
    Therefore, we have
    \begin{align*}
        \int \frac1{\abs{\bar{\Sigma}_n}^{\frac{\nu+n}2}} dx^n
        &= \frac{1}{\abs{\frac{\nu}{\nu+n}\bar\Sigma_0}^{\frac{\nu}2}}
        \prod_{i=1}^n \frac{
        \pi^{\frac{m}2} \Gamma(\frac{\nu+i-m}2)
        }{\rbr{\frac{\nu+i-1}{\nu+n}c_{i-1} }^{\frac{m}2} \Gamma(\frac{\nu+i}2)
        }
        \\&= \rbr{\frac{\nu+n}{\nu\sigma^2}}^{\frac{m\nu}2}
        \prod_{i=1}^n \frac{\Gamma(\frac{\nu+i-m}2) \pi^{\frac{m}2}
        }{ \rbr{\frac{1}{\nu+n}\frac{\rho^2\nu+i-1}{\rho^2\nu+i} }^{\frac{m}2}\Gamma(\frac{\nu+i}2)}
        \\&= \frac{\Gamma_m(\frac{\nu}2)\pi^{\frac{mn}2}(\nu+n)^{\frac{m(\nu+n)}{2}} \rbr{1+\frac{n}{\rho^2\nu}}^{\frac{m}2}
        }{\sigma^{m\nu}\Gamma_m(\frac{\nu+n}2)\nu^{\frac{m\nu}2}
        },
    \end{align*}
    which, combined with \eqref{eq:capacity_simple_reduction}, yields the formula of the capacity.
\end{proof}

Now, we derive an LNML with more flexible luckiness in terms of location and scale.
Since the space of normal distributions is closed with respect to affine transformations on data
and the capacity is invariant to any reparametrization of data,
\begin{align*}
    {\model'}^n &\eqdef
    \myset{f'(y^n;\mu,\Sigma)\pi(\mu, \Sigma;\nu, \sigma^2,\rho^2)}{
        \mu\in\real^m,\Sigma\in\real^{m\times m},\Sigma>0
    },
\end{align*}
where $y_i=A^{-1} x_i+\mu_0\;(0\le i<n)$ denotes a bijective affine mapping and $f'(;\mu,\Sigma)$
denotes the probability density function of $y^n$,
has the same capacity $C(\model'^n)=\Ccal(n,m,\nu,\sigma^2,\rho^2)$.
This yields the following theorem.

\begin{theorem}[Luckiness with location and scale parameters]
  \label{thm:LNML_location_scale}
    Let $\model^n$ be the multivariate normal distributions of $n$ observations
    with the luckiness given by
    \begin{align*}
        \pi(\mu, \Sigma; \nu, \mu_0, \Sigma_0, \rho^2)\eqdef
        \frac1{(2\pi)^{\frac{m\nu}2}|\Sigma|^{\frac{\nu}2}}\exp\sbr{
            -\frac{\nu}{2}\tr\cbr{\Sigma^{-1} \rbr{\Sigma_0+\rho^2(\mu-\mu_0)(\mu-\mu_0)^\top}}
        },
    \end{align*}
    where $\mu_0\in\real^{m}$ and $\Sigma_0(>0)\in\real^{m\times m}$.
    Therefore the capacity is calculated as
    \begin{align*}
        C(\model^n)=\mathcal{C}(n, m, \nu, |\Sigma_0|^{1/m}, \rho^2).
    \end{align*}
    The corresponding LNML is also given by
    \begin{align*}
        \bar{p}_n(x^n) &= \frac{
            f(x^n;\bar{\mu}_n,\bar\Sigma_n)
            \pi(\bar{\mu}_n,\bar\Sigma_n; \nu, \mu_0, \Sigma_0, \rho^2)
        }{\mathcal{C}(n, m, \nu, |\Sigma_0|^{1/m},\rho^2)}\\
        &=
        \rbr{
            \frac
            {\nu^{\nu}}
            {\pi^n(n+\nu)^{n+\nu}(1+\frac{n}{\rho^2\nu})}
        }^{\frac{m}2}
        \frac
        {\Gamma_m(\frac{n+\nu}2)}
        {\Gamma_m(\frac{\nu}2)}
        \frac{|\Sigma_0|^{\frac\nu2}}{|\bar\Sigma_n|^{\frac{n+\nu}2}}.
    \end{align*}
    where $\bar{\mu}_n=\mu_0+\frac{1}{n+\rho^2 \nu}\sum_{i=1}^{n} (x_i-\mu_0)$ and
    $\bar{\Sigma}_n=\frac1{n+\nu}(\sum_{i=1}^{n} (x_i-\mu_0)(x_i-\mu_0)^\top +\nu\Sigma_0)-\frac{1}{(\nu+n)(\rho^2\nu+n)}\sum_{i,j=1}^{n}(x_i-\mu_0)(x_j-\mu_0)^\top$
    denote the maximum a-posteriori probability (MAP) estimates under the prior proportional to $\pi$.
\end{theorem}
\begin{proof}
    Taking the scale $\sigma$ and the transformation $A$ such that $\Sigma_0=(\sigma A^{-1})(\sigma A^{-1})^\top$ and that $|A|=1$, we have
    \begin{align*}
      f'(y^n;\mu,\Sigma)\pi(\mu,\Sigma;\nu,\sigma^2, \rho^2)
      &=
      f(y^n;A^{-1}\mu+\mu_0,A^{-1}\Sigma (A^\top)^{-1})
      \pi(\mu,\Sigma;\nu,\sigma^2, \rho^2)
      \\&=
      f(y^n;\mu',\Sigma') \pi(A(\mu'-\mu_0), A\Sigma' A^\top; \nu, \sigma^2, \rho^2)
      \\&=
      f(y^n;\mu',\Sigma')\pi(\mu',\Sigma';\nu,\mu_0, \Sigma_0, \rho^2),
    \end{align*}
    where $\mu'=A^{-1}\mu+\mu_0$ and $\Sigma'=A^{-1}\Sigma (A^\top)^{-1}$.
    Therefore, since the transformation from $(\mu,\Sigma)$ to $(\mu', \Sigma')$ is bijective, we have
    \begin{align*}
      \model'^n=\myset{
        f(y^n;\mu',\Sigma')\pi(\mu',\Sigma';\nu,\mu_0, \Sigma_0, \rho^2)
      }{\mu'\in\RR^m,\Sigma'\in\RR^{m\times m},\Sigma'>0} = \model^n
    \end{align*}
    Then it follows that $C(\model^n)=C(\model'^n)=\Ccal(n,m,\nu,\sigma^2, \rho^2) =\Ccal(n,m,\nu,|\Sigma_0|^{1/m}, \rho^2)$.
\end{proof}

Note that the result of the last theorem can be applied to calculate the conditional NML (CNML)
for the multivariate normal distributions,
since $\pi(\mu,\Sigma;\nu,\mu_0,\Sigma_0,1)$ is proportional to the conjugate prior of the normal distributions.

As an immediate result of Theorem~\ref{thm:LNML_location_scale},
we have the sequential decomposition of the LNML.
\begin{corollary}[Sequential decomposition]
  \label{col:decomposition}
  The LNML distribution given in Theorem~\ref{thm:LNML_location_scale} is decomposed as
  \begin{align*}
    \bar p_n(x^n) 
    &= \prod_{i=1}^n t_{\nu-m+i}\rbr{x_i \;\middle|\; \bar\mu_{i-1},\frac{(\rho^2\nu+i)(\nu+i-1)}{(\rho^2\nu+i-1)(\nu-m+i)}\bar\Sigma_{i-1}},
  \end{align*}
  where $t_\nu(x|\mu,\Sigma)$ denotes the multivariate t-distribution with degree of freedom $\nu$, location $\mu$ and scale $\Sigma$.
\end{corollary}
Note that the decomposition given by Corollary~\ref{col:decomposition} is independent of the number of observations $n$.
Therefore, it defines an exchangeable stochastic process over positive integers.

\section{Discussion}
The hyperparameters $(\nu, \mu_0, \Sigma_0, \rho^2)$ can be chosen by the prior knowledge or by the nature of given data.
In either way, it does not matter to the asymptotic growth rate of the capacity
since the effect of luckiness function is no more than constant.
However, smaller $\nu$, e.g., $\nu=m$, is preferable if there is no specific reason,
because larger $\nu$ makes $\pi$ pointy and the associated regret $R_\pi$ more distorted.
On the other hand, $\Sigma_0$ should be smaller than the expected scale of the data distribution as it softly bounds the covariance estimate $\bar{\Sigma}_n$ from below in terms of eigenvalues.
As for the other hyperparameters, $\mu_0$ and $\rho^2$, we recommend to choose them such that the possible location of data distribution is included in the ball
$\myset{\mu\in\RR^m}{(\mu-\mu_0)^\top \Sigma_0^{-1}(\mu-\mu_0)\le \frac1{\nu\rho^2}}$.
Hence, for instance, a reasonable choice of luckiness is
$\pi(\mu, \Sigma; m, \sigma^2, \frac{\sigma^2}{m R^2})$ given by \eqref{eq:luckiness},
where $\sigma^2$ is an approximate lower bound of the minimum eigenvalue of the true covariance and $R$ is an approximate upper bound of the maximum norm of the true mean.
It remains for future work
to extend the result to other classes of distributions such as exponential families of distributions.

\bibliographystyle{IEEEtran}
\bibliography{sample}

\end{document}